\documentclass[11pt]{article}

\title{\vskip-1.0em\sc Directly finite algebras of pseudofunctions on locally compact groups}
\author{\sc Yemon Choi}
\date{2nd February 2014}
\date{13th May 2014}

\newcommand{\contact}{%
Department of Mathematics and Statistics\\
Fylde College\\
Lancaster University\\
Bailrigg, Lancaster\\
Lancashire LA1 4YF\\
Email: \texttt{y.choi1@lancaster.ac.uk}
}

\input{df-pfp_mac}

\numberwithin{equation}{section}

\usepackage{mathpazo}
\usepackage{sectsty}
\allsectionsfont{\sc}

\begin{document}
\maketitle

\begin{abstract}
An algebra $A$ is said to be directly finite if each left invertible element in the (conditional) unitization of $A$ is right invertible.
We show that the reduced group $\Cst$-algebra of a unimodular group is directly finite, extending known results for the discrete case.
We also investigate the corresponding problem for algebras of $p$-pseudofunctions,
showing that these algebras are directly finite if $G$ is amenable and unimodular, or unimodular with the Kunze--Stein property.

An exposition is also given of how existing results from the literature imply that $L^1(G)$ is not directly finite when $G$ is the affine group of either the real or complex line.

\medskip\noindent MSC 2010: 
22D15, 43A15 (primary); 22D25, 46L05 (secondary)
\end{abstract}


\begin{section}{Introduction}\label{s:intro}
This article arose from work on the following problem.
Let $G$ be a locally compact group, let $1<p<\infty$, and let $f\in C_c(G)$. By integrating $f$ against the left regular representation of $G$ on $L^p(G)$, we obtain a bounded linear operator $\lm_p(f):L^p(G)\to L^p(G)$.

\paragraph{Question.}
How big is the approximate point spectrum of $\lm_p(f)$\/?

\medskip\noindent
When $G$ is abelian, Gelfand/Fourier theory shows us that the spectrum of $\lm_p(f)$ consists entirely of approximate eigenvalues; that is, the approximate point spectrum is as big as possible. 
For non-abelian groups, new tools are needed, and we do not have a complete answer.

Nevertheless, when $G$ is discrete, it can be shown that the spectrum of $\lm_2(f)$ consists entirely of approximate eigenvalues. The key to proving this is an old observation, apparently first noted by Kaplansky, that the group algebra of a \emph{discrete} group is directly finite, as defined in the abstract. (See \cite[Remark 2.10]{YC_surjunc} for some further remarks on the history of this observation, and later proofs.)
It is therefore natural to ask if the same is true for an arbitrary locally compact group, and this leads to the first main result of this article.

\begin{thm}\label{t:PF2_unimodular}
If $G$ is unimodular, then its reduced group $\Cst$-algebra is \qdf.
\end{thm}

The proof boils down to an extension of the technique used in 
\cite{Mon_dirfin} for discrete groups, with the ``Kaplansky trace'' replaced by the ``Plancherel weight''. Note, however, that Theorem~\ref{t:PF2_unimodular} does not work if we replace the reduced group $\Cst$-algebra with the group von Neumann algebra. For instance, if $G$ is the real Heisenberg group, then its group von Neumann algebra is isomorphic to $L^\infty(\Real)\overline{\otimes} \Bdd(L^2(\Real))$ and therefore contains a large supply of left-invertible elements that are not invertible.
The same phenomenon also occurs for semisimple Lie groups that are not compact.

By repeating the arguments used in the author's previous article \cite{YC_surjunc} (specifically, the proof of Theorem 1.2 in that paper), one obtains the following corollary of Theorem~\ref{t:PF2_unimodular}, giving an answer to certain cases of the question raised at the start.

\begin{cor}\label{c:all-approx-eval}
Let $G$ be unimodular and let $f\in C_c(G)$. Then the spectrum of the operator $\lm_2(f): L^2(G)\to L^2(G)$ consists entirely of approximate eigenvalues.
\end{cor}

\begin{rem}
The arguments from  \cite{YC_surjunc} would actually give slightly more, namely that Corollary~\ref{c:all-approx-eval} remains true if $L^2(G)$ is replaced by any noncommutative $L^p$-space associated to the group von Neumann algebra $\VN(G)$. We omit the details, since this article is concerned with operators on $L^p(G)$ rather than on $L^p(\VN(G))$.
\end{rem}

In attempting to establish the same result for $\lm_p(f)$ where $p\neq 2$, a natural first step would be to show that $\PF_p(G)$, the algebra of~\dt{$p$-pseudofunctions on~$G$}, is also directly finite. We were unable to obtain a proof or a counter-example for general unimodular groups, but by building on the $\Cst$-case we obtain partial results.
The following result follows easily 
from combining Theorem~\ref{t:PF2_unimodular} with an embedding result of Herz.

\begin{thm}\label{t:PFp_amenable}
Let $G$ be an amenable, unimodular, locally compact group. Then $\PF_p(G)$ is directly finite for all $1<p<\infty$.
\end{thm} 

Results of Lohou\'e \cite{Lohoue_1980} show that the assumption of amenability in Herz's embedding result is necessary, even in if one restricts to unimodular examples.
Nevertheless, for certain non-amenable groups one can use a similar idea to obtain the same conclusion.

\begin{thm}\label{t:PFp_KS}
Let $G$ be a unimodular, locally compact group which has the \dt{Kunze--Stein} property (see Section~\ref{s:PFp-is-qdf} below for the definition). Then $\PF_p(G)$ is directly finite for all $1<p<\infty$.
\end{thm} 

Examples of groups with the Kunze--Stein property are: semisimple Lie groups with finite centre, and certain automorphism groups of trees, in their natural totally disconnected topology. We note that a group is both amenable and Kunze--Stein if and only if it is compact.

It is tempting to try to remove the condition of unimodularity. However, such hopes are dashed by the following examples.

\begin{thm}\label{t:L1-not-df}
Let $G$ be the affine group of either $\Real$ or $\Cplx$. Then $L^1(G)$ (and hence each of the algebras $\PF_p(G)$ for $1<p<\infty$) fails to be directly finite.
\end{thm}

\begin{rem}
The key to the proof of Theorem~\ref{t:PF2_unimodular} is the existence of a densely-defined, faithful trace on $\Cst_r(G)$ when $G$ is unimodular. Such a trace no longer exists when $G$ is the affine group of $\Real$ or $\Cplx$. However, $\VN(G)$ is in both cases \emph{semifinite}, and therefore has an \emph{ultraweakly-densely-defined} faithful trace. This should be borne in mind when we come to the proof of Theorem~\ref{t:PF2_unimodular}.
\end{rem}

Theorem~\ref{t:L1-not-df} can be derived quickly by combining some known results scattered through the literature.
Namely, one combines a theorem of Leptin, which says that the inclusion $L^1(G)\to\Cst_r(G)$ is spectrum-preserving when $G$ is either of the groups described above, with calculations of Diep and of Rosenberg that were used in \emph{classifying} $\Cst_r(G)$ for certain solvable Lie groups~$G$.
Since Theorem~\ref{t:L1-not-df} seems to have gone unrecorded, as far as the author was able to find, and since the results used in its proof may be of independent interest to researchers in abstract harmonic analysis and Banach algebras, some details are included in an expository section.

Let us briefly describe the structure of this paper. We set up some preliminary notation and observations on direct finiteness in Section~\ref{s:df-BA}, with emphasis on how to formulate the concept in a way that works equally well for rings with or without identities. In Section~\ref{s:df-cstar} we specialize to the setting of $\Cst$-algebras equipped with densely-defined faithful traces, and show how the ideas of \cite{Mon_dirfin} can be extended beyond the setting of discrete groups to yield a proof of Theorem~\ref{t:PF2_unimodular}. Section~\ref{s:PFp-is-qdf} initiates a study of direct finiteness for the algebras $\PF_p(G)$ in the cases $p\neq 2$, and gives the proof of Theorems~\ref{t:PFp_amenable} and~\ref{t:PFp_KS}. Section~\ref{s:not-qdf} is essentially expository, but shows how Theorem~\ref{t:L1-not-df} is proved; and in the final section we pose some questions for future work. In an appendix we show how to extract from two dense papers of Herz the bare minimum needed for the proof of Theorem~\ref{t:PFp_amenable}.

One aim of this paper is to give yet another demonstration that fairly simple operator-algebraic techniques can be used profitably in abstract harmonic analysis. We have attempted to make it accessible to those working in Banach algebras or non-abelian harmonic analysis who, like the author, are not operator algebraists. In several places this has meant repeating some material which is well-known to specialists, in order to provide extra background, or to make the article more self-contained, or to provide references to sources where the reader can find actual proofs (as opposed to assertions or references to proofs). We hope that the sacrifice of brevity for exposition will not trouble the reader unduly.
\end{section}

\begin{section}{Directly finite Banach algebras}
\label{s:df-BA}
Throughout, we adopt the convention that a ring need not be commutative, nor contain an identity element.
A ring $R$ with identity is said to be \dt{directly finite} if each left-invertible element of $R$ is invertible. Motivated by examples from semigroup theory, Munn \cite{Munn_DF1} generalized this definition to the setting of rings without identity. It is convenient to present Munn's definition using the notions of left and right quasi-inverses.

\begin{dfn}
Let $R$ be a ring. Given $a,b\in R$, let $a\qm b\defeq a+b-ab$. If $a\qm b=0$ then we say that $a$ is a \dt{left quasi-inverse} for $b$ and $b$ is a \dt{right quasi-inverse} for $a$\/. An element which has both a left and a right quasi-inverse is said to be \dt{quasi-invertible}.
\end{dfn}

Of course, if $R$ has an identity element $\id$, then
\begin{equation}\label{eq:dictionary}
\id-a\qm b = (\id-a)(\id-b).
\end{equation}
and it is clear from this, or directly from the definition, that $a\qm 0=a = 0\qm a$ for all $a\in A$.
Moreover, the operation $\qm$ is associative: one could check this by a direct calculation, but it is more instructive to adjoin a formal identity $\id$ and observe, using \eqref{eq:dictionary} repeatedly, that
\[ \begin{aligned}
\id-(a\qm b)\qm c = (\id-a\qm b)(\id-c)
 & = (\id-a)(\id-b)(\id-c) \\
 & = (\id-a)(\id-b\qm c) = \id-a\qm (b\qm c)
\end{aligned} \] 
 (In the Banach-algebraic setting, a convenient reference for all this is~\cite[\S4]{BonsDunc}.)

\begin{rem}
It is more intuitive to reason with left, right and two-sided invertible elements than with their ``quasi-'' counterparts. On the other hand, the language of quasi-inverses sidesteps the clunky use of conditional unitizations, allowing us to treat unital and non-unital cases together.
\end{rem}

\begin{dfn}\label{d:dirfin}
A ring $R$ is \dt{directly finite}
if any pair $(a,b)\in R\times R$ satisfying $a\qm b= 0$ also satisfies $b\qm a = 0$.
\end{dfn}

\begin{rem}
In \cite{Munn_DF1}, a ring with this property is said to be ``quasidirectly finite''. We have chosen instead to extend the terminology from the unital case to the non-unital one.
\end{rem}

It is clear that if $S$ is a directly finite ring and $\phi:R\to S$ is an injective ring homomorphism, $R$ is directly finite; this holds even if $R$ and $S$ have identity elements and $\phi(\id[A])\neq \id[S]$.

We now discuss unitizations, and from here on switch to algebras rather than rings (just to sidestep the annoyance that the ring-unitization of an algebra is not its algebra-unitization).
Recall that the \dt{forced unitization} of a $k$-algebra $A$, which we shall denote by $\fu{A}$, is defined to be the vector space $A\oplus k$, equipped with the multiplication $(a,\lm)(b,\mu)\defeq (ab+\lm b+ \mu a, \lm\mu)$, and with the identity element $(0,1)$ denoted by~$\id$.

\begin{lem}\label{l:unital}\
\begin{YCnum}
\item Let $A$ be an algebra (with or without identity). Then $A$ is directly finite (in the sense of Definition~\ref{d:dirfin}) if and only if each left-invertible element of $\fu{A}$ is invertible.
\item Let $A$ be an algebra with identity. Then $A$ is directly finite (in the sense of Definition~\ref{d:dirfin}) if and only if each left-invertible element of $A$ is invertible.
\end{YCnum}
\end{lem}

The lemma follows easily from the identity \eqref{eq:dictionary}; we omit the proof.

In a unital Banach algebra, the group of invertible elements is open in the norm topology. There is an analogous result for quasi-inverses, see \cite[Theorem 4.8]{BonsDunc}; we shall require a slightly more precise version.

\begin{lem}\label{l:stable}
Let $A$ be a Banach algebra and let $c\in A$. Suppose there exists $b\in A$ with $b\qm c=0$. Then, for each $c'\in A$ that is sufficiently close to $c$, there exists $a\in A$ such that $a \qm b\qm c' = 0$.
\end{lem}

\begin{proof}
Give the unitization $\fu{A}$ its usual norm, so that $\norm{\id}=1$.
We have $(\id-b)(\id-c)=\id$ in $\fu{A}$. Let $\delta=(1+\norm{b})^{-1}>0$ and suppose $c'\in A$ satisfies $\norm{c'-c} < \delta$\/.
Put $u \defeq (\id-b)(\id-c') \in \fu{A}$.
Since $\norm{\id-u} = \norm{ (\id-b)(c'-c)} <1$, $u$ is invertible in the Banach algebra $\fu{A}$, and $a \defeq \id-u^{-1} = - \sum_{n\geq 1} (\id-u)^n$
lies in $A$. By construction, $\id-u= b\qm c'$, so that $a\qm b\qm c'=0$ as required.
\end{proof}

The following proposition will be important in the next section, where we only apply it to certain $2$-sided ideals, and also in Section~\ref{s:PFp-is-qdf}, where we really do need to consider left ideals as well.
The author does not know if the same result remains true under the weaker hypothesis that $J$ is merely a dense \emph{subalgebra} of~$A$.

\begin{prop}\label{p:qdf-ideal}
Let $A$ be a Banach algebra, and let $J$ be a \emph{left} ideal in $A$ which is dense for the norm topology. Then $J$ is \qdf\ if and only if $A$~is.
\end{prop}

\begin{proof}
Clearly, if $A$ is \qdf\ then so is $J$. Conversely, suppose $J$ is \qdf, and let $b,c\in A$ satisfy $b\qm c=0$.
By Lemma~\ref{l:stable} and density of $J$ in $A$, we can find $c'\in J$ and $a\in A$ such that $a\qm b \qm c' = 0$.

Put $b'=a\qm b$; then since $b'\qm c'=0$, we have $b' = b'c'- c' \in J$, as $c'\in J$ and $J$ is a left ideal. Since $J$ is \qdf, $c'\qm b'=0$, that is, $c'\qm a\qm b=0$. So $b$ has a left quasi-inverse in $A$, and as $b\qm c=0$ we conclude that $c\qm b=0$ as required.
\end{proof}

\end{section}

\begin{section}{Directly finite $\Cst$-algebras, via densely-defined traces}
\label{s:df-cstar}
We recall some standard $\Cst$-algebraic terminology from \cite[\S5]{Ped_cstar-book}, just to fix the notation.
If $A$ is a $\Cst$-algebra, $A_+$ will denote its cone of positive elements. A~\dt{weight on $A_+$} is a function $\phi: A^+\to [0,\infty]$ that is $\Real_+$-linear and additive.
$\phi$~is \dt{faithful} if $\phi(x)>0$ for all $x\in A_+\setminus\{0\}$; it is
\dt{tracial} if it satisfies $\phi(u^*au)=\phi(a)$ for all $a\in A_+$ and all unitary $u\in \fu{A}$.

Given a tracial weight $\tau$ on $A_+$, there exists a $2$-sided $*$-ideal $A^\tau \subseteq A$,
and a linear tracial functional $A^\tau\to \Cplx$ which coincides with $\tau$ on $A^\tau \cap A^+$; by abuse of notation, we will denote this functional also by~$\tau$. $A^\tau$~is called the \dt{ideal of definition of the trace~$\tau$.}
Moreover, the set 
\[ A^\tau_2\defeq \{x \in A : \tau(xx^*) <\infty\}\]
has the same norm-closure in $A$ as does~$A^\tau$.
(See \cite[\S5.1]{Ped_cstar-book} for the details.)

\begin{prop}\label{p:df-cstar}
Let $A$ be a $\Cst$-algebra and $\tau$ a faithful tracial weight~on $A_+$. Let $B$ denote the norm-closure of $A^\tau$ inside $A$. Then $B$ is \qdf.
\end{prop}

When $A$ is unital and $A^\tau= A=B$, this result is known, and gives a quick proof that the group von Neumann algebra of a discrete group is directly finite (see~\cite{Mon_dirfin}) which does not require classification of projections.
The author is unaware of a reference which explicitly states Proposition~\ref{p:df-cstar} in the generality given here, so a full proof will be given below.
The key observation is the following standard result about $\Cst$-algebras.

\begin{lem}\label{l:idempotent-folklore}
If $p$ is an idempotent in a $\Cst$-algebra, there exists a hermitian idempotent $e$ in that algebra which satisfies $ep=p$ and $pe=e$.
\end{lem}

When discussing this work with other researchers, the author has found Lemma~\ref{l:idempotent-folklore} is not universally known. For convenience, we give an outline of the standard proof.

 \begin{proof}[Proof of Lemma~\ref{l:idempotent-folklore}]
After adjoining an identity element to our algebra if necessary, we make the {\it Ansatz}
\begin{equation}\label{eq:rabbit}
\tag{$\dagger$} e= pp^*(\id+(p-p^*)(p^* - p))^{-1}\,,
\end{equation}
which clearly satisfies $e=pe$. It remains to check that $ep=p$ and that $e$ is an idempotent. One could verify this directly, but a more intuitive approach is to take a faithful representation of our (unital) $\Cst$-algebra on a Hilbert space~$H$, and regard $p$ as a $2\times 2$ operator matrix
$\left(\begin{matrix} I & R \\ 0 & 0 \end{matrix} \right)$
with respect to the decomposition $H=\ran(p)\oplus \ran(p)^\perp$. Computing the right-hand side of \eqref{eq:rabbit} with respect to this block-matrix decomposition gives
$\left(\begin{matrix} I & 0 \\ 0 & 0 \end{matrix} \right)$, which is the orthogonal projection of $H$ onto $\ran(p)$.
\end{proof}

\begin{rem}
A similar but less spatial approach, taken by Kaplansky himself in \cite{Kap_AJM53}, is to exploit the theory of polynomial identities; the author thanks J. Meyer \cite{MO16944} for bringing this to his attention.
\end{rem}

\begin{proof}[Proof of Proposition~\ref{p:df-cstar}]
By Proposition~\ref{p:qdf-ideal}, it suffices to show that $A^\tau$ is \qdf.
Let $a,b\in A^\tau$ be such that $a\qm b=0$\/, and let $p \defeq b\qm a = ab-ba\in A^\tau$\/. We must show $p=0$.
Observe that $p^2=p$, since
\[ 2p-p^2 = p\qm p = b\qm a\qm b \qm a = b\qm 0 \qm a = p \,.\]
By Lemma~\ref{l:idempotent-folklore}, there is a self-adjoint idempotent $e\in B$ such that $pe=e$ and $ep=p$\/; in particular, $e\in A^\tau$, since $A^\tau$ is a (right) ideal in~$B$. Since $\tau$ is tracial,
\[ \tau(e^*e)= \tau(e) = \tau(pe) = \tau(ep)=\tau(p)=\tau(ab)-\tau(ba)=0\,.\]
As $\tau$ is faithful, this forces $e=0$. Hence $p=0$ as required.
\end{proof}

\medskip
Proposition~\ref{p:df-cstar} applies, in particular, to the norm-completions of Hilbert algebras.
Let us recall some of the terminology for ease of reference. A \dt{Hilbert algebra} is an associative, complex $*$-algebra~$\fA$,
equipped with an inner product $\ip{\cdot}{\cdot}$ that satisfies certain compatibility conditions: see \cite[ch.~I\S5, d\'efn~1]{Dix_VN_fr2} or~\cite[Definition 11.7.1]{Palmer2}.
Completing $\fA$ in the norm $x\mapsto \ip{x}{x}^{1/2}$ yields a Hilbert space $\cH$.
Given $a\in \fA$ we denote by $U_a:\cH\to \cH$ the unique operator satisfying $U_a(x)=ax$ for all $x\in\fA$; the map $\lambda:a\mapsto U_a$ is an injective $*$-homomorphism from $\fA$ into $\Bdd(\cH)$.%
(See \cite[\S11.7]{Palmer2} for a summary of these results, or \cite[ch.~I,\S\S5--6]{Dix_VN_fr2} for full details and proofs.)

The closure of $\lambda(\fA)$ in the weak operator topology of $\Bdd(\cH)$ is denoted by $\cU(\fA)$, and called the \dt{left von Neumann algebra assoc\-iated to~$\fA$}.
A fundamental result in the theory of Hilbert algebras tells us there is a faithful tracial weight $\phi$ on $\cU(\fA)_+$, such that
\begin{equation}\label{eq:bounded}
\phi(U_a^*U_a)<\infty \qquad\text{ for all $a\in\fA$.}  \tag{$*$}
\end{equation}
(See \cite[ch.~I,\S6]{Dix_VN_fr2}, proposition 1 and th\'eor\`eme~1 for the proof.)

\begin{cor}\label{c:hilbert-algebra-is-df}
Let $\fA$ be a Hilbert algebra, and let
$B$ be the norm-closure of~$\lm(\fA)$ inside $\Bdd(\cH)$.
 Then $B$ is \qdf.
\end{cor}

\begin{proof}
Let $\tau$ be the restriction of the weight $\phi$ to $B_+$; then $\tau$ is a faithful tracial weight on $B_+$, and in the notation used at the start of this section, \eqref{eq:bounded} tells us that $\lm(\fA)\subseteq B^\tau_2$. Hence $B^\tau_2$ is dense in~$B$, and by our earlier general remarks on tracial weights, this implies $B^\tau$ is dense in $B$. Applying Proposition~\ref{p:df-cstar} completes the proof.
\end{proof}

We can now prove Theorem~\ref{t:PF2_unimodular}. 
A standard example of a Hilbert algebra is given by the space $C_c(G)$ of continuous, compactly supported functions on~$G$, equipped with convolution as product, when $G$ is \emph{unimodular}.
(For full details of the Hilbert algebra structure, see~\cite[chapitre 1, \S5, exercice~5]{Dix_VN_fr2} and the ensuing hints, or \cite[\S13.10]{Dix_C*_en}, or \cite[Example 11.7.2]{Palmer2}.) The associated Hilbert space $\cH$ is just $L^2(G)$; the representation $\lm: \fA\to \Bdd(\cH)$ is the left regular representation $\lm_2$ defined earlier; 
and the norm-closure of $\lm(\fA)$ inside $\Bdd(\cH)$ is~$\Cst_r(G)$.
Therefore Theorem~\ref{t:PF2_unimodular} follows as a special case of Corollary~\ref{c:hilbert-algebra-is-df}.

\begin{rem}
We proved Theorem~\ref{t:PF2_unimodular} by using the existence of a densely-defined faithful trace on $\Cst_r(G)$, but we appealed to rather general results to produce this trace. It may be instructive to have a more explicit description of what is going on.  Let $\FA(G)$ be the \dt{Fourier algebra} of~$G$, and define
\[ J = \{ f\in \FA(G) \st h\mapsto f*h \text{ is bounded from $C_c(G)$ to $L^2(G)$}\}, \]
which we might describe somewhat loosely as $\FA(G)\cap \VN(G)$. The faithful trace $\tau$ can then be described on $J$ as $\tau(f)=f(e)$, where $e$ is the identity element of the group; this is the \dt{Plancherel weight} on~$\VN(G)$.
Finally, $J$ contains $\FA(G)\cap C_c(G)$, so $J\cap \Cst_r(G)$ is the desired dense ideal in $\Cst_r(G)$. (In general $J\not\subseteq \Cst_r(G)$; for instance, if $G=\Real$ then the Fourier transform intertwines $J$ with $(L^1\cap L^\infty)(\widehat{\Real})$ and $\Cst_r(\Real)$ with $C_0(\widehat{\Real})$.)
\end{rem}

It is natural to ask if the converse of Theorem~\ref{t:PF2_unimodular} holds. The following examples show it does not, although in some sense they constitute a `cheat answer'.

\begin{eg}[A certain family of solvable Lie groups]\label{eg:cheat}
Let $p$, $q$ be strictly positive integers, and let $\al=(\al_1,\dots,\al_{p+q})$ where each $\al_j$ is a strictly positive real number. Denote by $G(p,q,\al)$ the semidirect product $\Real^{p+q} \rtimes \Real$, where the action of $\Real$ on $\Real^{p+q}$ is given by
\[ t \mapsto \operatorname{diag}(e^{\al_1 t},\dots, e^{\al_p t},
	e^{-\al_{p+1} t}, \dots, e^{-\al_{p+q} t}) \]
In general, $G(p,q,\al)$ is not unimodular. However, 
the isomorphism class of $\Cst(G(p,q,\al))$ depends only on $p$ and $q$ and not on $\al$. (See \cite[pp.~12--13]{Wang_pitman199} and \cite[pp.~190--191]{Ros_Pac76}.) 
Moreover, if we happen to choose $\al$ such that 
 $\al_1+\dots+\al_p=\al_{p+1}+\dots+\al_{p+q}$, then $G(p,q,\al)$ \emph{is} unimodular; see~\cite[p.~190]{Ros_Pac76}.
Therefore, by  Theorem~\ref{t:PF2_unimodular}, $\Cst_r(G(p,q,\al))$ is \qdf\ even when $G(p,q,\al)$ is non-unimodular.
\end{eg}

\begin{rem}\label{r:not-SOT-closure}
In general, if $\fA$ is a Hilbert algebra, $\cU(\fA)$ need not be directly finite.
For instance, take $\fA$ to be the $*$-algebra of finite rank operators on $\ell^2$, equipped with the Hilbert algebra structure defined by the bilinear map $(S,T) \mapsto \operatorname{tr}(TS)$.
Then the associated Hilbert space can be identified with $\ell^2$, and $\cU(\fA)=\Bdd(\ell^2)$, which is clearly not \qdf\ (just look at any non-unitary isometry on $\ell^2$).
If we want examples of the form $C_c(G)$, then as mentioned in the introduction, we could take $G$ to be the real Heisenberg group or any semisimple connected Lie group; in all such cases $\VN(G)$ will contain an isomorphic copy of $\Bdd(\ell^2)$, but since $G$ is unimodular $\Cst_r(G)$ will be directly finite.
\end{rem}

\end{section}

\begin{section}{Some groups for which $\PF_p(G)$ is \qdf}
\label{s:PFp-is-qdf}
We start with a quick review of the algebras $\PF_p(G)$.
Let $G$ be a locally compact group, for now not necessarily unimodular, and as before regard $C_c(G)$ as an algebra with product given by convolution.
For each $p\in (1,\infty)$, define a homomorphism $\lm_p:C_c(G) \to \Bdd(L^p(G))$ by
\[ \lm_p(f)(h) = f*h \qquad(\text{$f\in C_c(G)$, $h\in L^p(G)$}). \]
The norm-closure of $\lm_p(C_c(G))$ inside $\Bdd(L^p(G))$ is denoted by $\PF_p(G)$ and called the \dt{algebra of $p$-pseudo\-functions on~$G$}. Note that $\PF_2(G)$ is nothing but the reduced group $\Cst$-algebra of~$G$.

We have deliberately avoided introducing general convolution operators on $L^p(G)$ and how one represents them by (possibly infinite) Radon measures on~$G$, just to keep the technicalities to a minimum.
Further references and additional details are given in the monograph~\cite{Der_book}, although some of the basic properties are stated without full proofs.

\begin{rem}\label{r:half-suffices}
Let $G$ be a unimodular, locally compact group.
Let $p$ and $q$ be conjugate indices, strictly between $1$ and $\infty$. The operator $\Bdd(L^p(G)) \to \Bdd(L^q(G))$, $T\mapsto T^*$, is an anti-isomorphism from $\PF_p(G)$ onto $\PF_q(G)$. It follows that $\PF_q(G)$ is \qdf\ if and only if $\PF_p(G)$~is.
\end{rem}

One would like to generalize Theorem~\ref{t:PF2_unimodular} to cover $\PF_p(G)$ for $G$ unimodular and all $p\in (1,\infty)$, or to find a unimodular example for which $\PF_p(G)$ is not directly finite for some $p\neq 2$. This section presents some partial results in the positive direction.
We start by considering the case where $G$ is amenable.
The following theorem is a special case of results of Herz.

\begin{thm}[Herz]\label{t:pfp-in-pf2}
Let $G$ be a locally compact, amenable group. Regard $\lm_p$ as a homomorphism from $C_c(G)$ to $\PF_p(G)$. Then there is an injective homomorphism
$J:\PF_p(G)\to\PF_2(G)$
 such that  $\lm_2= J\lm_p$.
\end{thm}

Combining Theorem~\ref{t:pfp-in-pf2} with Theorem~\ref{t:PF2_unimodular}, the following is immediate.

\begin{cor}\label{c:PF_p-via-Herz}
Let $G$ be an amenable, unimodular, locally compact group. Then $\PF_p(G)$ is \qdf, for every $1<p<\infty$.
\end{cor}

\begin{rem}\label{r:what-Herz-did}
Let $\CV_p(G)$ denote the subalgebra of $\Bdd(L^p(G))$ consisting of all operators that commute with right translations.
Herz's full result, proved by combining results in \cite{Herz_p-space} and \cite{Herz_AIF73}, says that when $G$ is amenable
there is a unital embedding $\CV_p(G)\to\VN(G)$
 that extends the homomorphism $\lm_2:C_c(G)\to \VN(G)$.
Clearly this contains Theorem~\ref{t:pfp-in-pf2} as a special case. However, extracting a complete proof of the general result from these dense papers requires some work, since the necessary results are intertwined with other technically demanding results that are superfluous in the present context.
Therefore, in Appendix~\ref{app:Herz-by-hand}, we have given a quick account of those results from Herz's papers needed to prove Theorem~\ref{t:pfp-in-pf2}.
\end{rem}


Now we change direction and turn away from amenable groups.
\begin{dfn}
\label{d:Kunze--Stein}
A locally compact group $G$ is said to have the~\dt{Kunze--Stein property}, or to be a \dt{Kunze--Stein group}, if for each $1\leq p < 2$
there exists a constant $C_p\geq 1$ such that
\begin{equation}\label{eq:KS-ineq}
\norm{g*h}_2 \leq C_p \norm{g}_p \norm{h}_2 \quad\text{for all $g\in L^p(G)$ and all $h\in C_c(G)$.}\end{equation}
\end{dfn}

\begin{eg}[Examples with the Kunze--Stein property]
\
\begin{YCnum}
\item Let $G$ be a connected semisimple Lie group with finite centre, such as $\SL(n,\Real)$. Cowling~\cite{Cowling_KS78} proved that $G$ has the Kunze--Stein property.
(The particular case $n=2$ was established by Kunze and Stein in \cite{KunzeStein_AJM60}.)
\item Let $T$ be a homogeneous tree of order $\geq 3$, and equip $\Aut(T)$ with the topology of pointwise convergence; this makes it into a locally compact, totally disconnected group. Let $G$ be a closed subgroup of $\Aut(T)$ which acts transitively on the boundary~$\partial T$. (For instance, when ${\mathbb K}$ is a local field, the group $\SL(2,{\mathbb K})$ is of this form.) Then $G$ has the Kunze--Stein property~\cite{Nebbia_KS88}.
\end{YCnum}
\end{eg}

It is also remarked in \cite{Cowling_KS78} that the only amenable groups with the Kunze--Stein property are the compact ones.
In view of this, the following result is a somewhat surprising counterpart to Corollary~\ref{c:PF_p-via-Herz}.

\begin{thm}\label{t:PFp(KS)}
Let $G$ be unimodular and Kunze--Stein. Then $\PF_p(G)$ is \qdf.
\end{thm}

The proof of Theorem~\ref{t:PFp(KS)} occupies the rest of this section.
The main idea is to define a dense, one-sided ideal in $\PF_p(G)$ which is also a subalgebra of $\PF_2(G)$, and then use 
Theorem~\ref{t:PF2_unimodular}. 
Let $\diag_p: C_c(G) \to \LplusPG{p}$ be the `diagonal' embedding, i.e.~$\diag_p(f) = (f,\lm_p(f))$,
and define $X_p(G)$ to be the norm closure of $\diag_p(C_c(G))$.
Let $\pi_L$ and $\pi_P$ denote the coordinate projections from $\LplusPG{p}$ onto $L^p(G)$ and $\PF_p(G)$ respectively

\begin{lem}
The restrictions of $\pi_L$ and $\pi_P$ to the subspace $X_p(G)$ are both injective.
\end{lem}

\begin{proof}
If $(g,T)\in X_p(G)$ and $h\in C_c(G)$, an easy continuity argument shows that $g*h=T(h) \in L^p(G)$. Moreover, if $S\in \Bdd(L^p(G))$ and $S(h)=0$ for all $h\in C_c(G)$, then $S=0$. The result follows.
\end{proof}

\begin{lem}\label{l:dense-ideal-in-PFp}\
\begin{YCnum}
\item\label{li:left-ideal} Let $S\in \PF_p(G)$ and $(g,T)\in X_p(G)$. Then $(S(g),ST)\in X_p(G)$.
\item\label{li:BA} The bilinear map $X_p(G)\times X_p(G) \to X_p(G)$, defined by
\[
((f,S) , (g,T)) \mapsto (S(g),ST)
\]
is associative, and makes $X_p(G)$ a Banach algebra; moreover, $\pi_P:X_p(G) \to \PF_p(G)$ is an algebra homomorphism with dense range.
\end{YCnum}
\end{lem}

\begin{proof}
For this proof, we denote the usual $L^p$-norm on $C_c(G)$ by $\norm{\cdot}_p$ and the operator norm on $\Bdd(L^p(G))$ by $\norm{\cdot}_{p\to p}$.
Let $(g_n)$ be a sequence in $C_c(G)$ with $\diag(g_n)\to (g,T)$; that is,
$\norm{g_n-g}_p \to 0$ and $\norm{\lm_p(g_n)-T}_{p\to p} \to 0$.
Let $(f_n)$ be a sequence in $C_c(G)$ with $\norm{\lm_p(f_n)-S}_{p\to p} \to 0$. Then $f_n*g_n \in C_c(G)$; and a standard ``$3\veps$ argument'' shows that $\norm{f_n*g_n \to S(g)}_p \to 0$ and $\norm{\lm_p(f_n*g_n) - ST}_{p\to p} \to 0$.
This proves part~\ref{li:left-ideal}.
For part~\ref{li:BA}: associativity can be checked directly, as can the fact that the norm on $X_p(G)$ is submultiplicative. So $X_p(G)$ is a Banach algebra. Clearly $\pi_P:X_p(G) \to \PF_p(G)$ is a homomorphism; it has dense range, since $\pi_p\diag_p=\lm_p$.
\end{proof}

\begin{rem}
Morally speaking, $X_p(G)=\PF_p(G)\cap L^p(G)$; however, as we have defined things, $\PF_p(G)$ is a space of operators on $L^p(G)$ and not a space of functions or measures on $G$. Rather than set up enough machinery to make sense of elements of $L^p(G)$ as densely-defined convolution operators on $L^p(G)$, or to regard elements of $\PF_p(G)$ as certain kinds of distributions on~$G$, it seemed easier here to use ``soft'' tools.
\end{rem}

\begin{lem}\label{l:KS-diagram}
When $G$ is Kunze--Stein, there exist injective, continuous linear maps $\iKS:L^p(G)\to \PM_2(G)$ and $\tillm_{p,2}: X_p(G) \to \PF_2(G)$ making the following diagram commute:
\begin{equation}\label{eq:KS-diagram}
\begin{diagram}[tight,height=3em,width=2.5em]
  & & & & & & L^p(G) \\
  & & & & &\ruTo(6,2)^{\imath_p}   \ruTo_{\pi_L} & \\
C_c(G) & & \rTo_{\diag_p} & & X_p(G) & & \dTo_{\iKS}  \\
  & \rdTo(6,2)_{\lm_2} & & & & \rdTo^{\tillm_{p,2}} & \\
  & & &  & & & \PF_2
\end{diagram}
\end{equation}
Moreover, $\tillm_{p,2}$ is an algebra homomorphism.
\end{lem}

\begin{proof}
Since $G$ is Kunze--Stein, it follows from \eqref{eq:KS-ineq} that
there exists a bounded linear map $\iKS: L^p(G) \to \PF_2(G)$, which makes the outer triangle in Diagram~\eqref{eq:KS-diagram} commute.
If $f\in \ker\iKS$ then $f*h=0$ for each $h \in C_c(G)$, so by basic measure theory $f=0$ as an element of $L^p(G)$.
Thus $\iKS$ is injective.

Put $\tillm_{p,2} \defeq \iKS\pi_L$; clearly this is continuous and linear, and it is injective since $\iKS$ and $\pi_L$ are. By construction , it makes the
 right-hand inner triangle in \eqref{eq:KS-diagram} commute.
Now the top inner triangle in \eqref{eq:KS-diagram} commutes, by the definitions of $\diag_p$ and $\pi_L$. Hence, by a straightforward diagram chase, the remaining inner triangle commutes.

Finally, since $\tillm_{p,2}\diag_p = \lm_2$ is a homomorphism and $\diag_p$ has dense range, it follows by continuity that $\tillm_{p,2}$ is a homomorphism.
\end{proof}

\begin{proof}[Proof of Theorem~\ref{t:PFp(KS)}]
Let $G$ be unimodular and Kunze--Stein.
For $p=2$, Theorem~\ref{t:PFp(KS)} is a special case of Theorem~\ref{t:PF2_unimodular}. So by Remark~\ref{r:half-suffices}, it suffices to consider the case $1\leq p <2$.

By Lemma~\ref{l:KS-diagram}, $\tillm_{p,2}:X_p(G) \to \PF_2(G)$ is an \emph{injective} algebra homomorphism; thus $X_p(G)$ is \qdf, using Theorem~\ref{t:PF2_unimodular}.
Consider $\pi_P(X_p(G))\subseteq \PF_p(G)$; this is \qdf, since $\pi_P$ is an \emph{injective} algebra homomorphism. On the other hand, by Lemma~\ref{l:dense-ideal-in-PFp}, $\pi_P(X_p(G))$ is a dense left ideal in $\PF_p(G)$.
Therefore, by Proposition~\ref{p:qdf-ideal}, $\PF_p(G)$ is \qdf.
\end{proof}

\end{section}

\begin{section}{Two groups for which $L^1(G)$ is not \qdf}
\label{s:not-qdf}
Let $\bbF$ be either $\Real$ or $\Cplx$, equipped with its usual topology: write $\bbF^\times$ for $\bbF\setminus\{0\}$, regarded as a multiplicative group. The \dt{affine group of~$\bbF$}, denoted by $\Aff(\bbF)$, is defined to be the group
\[ \left\{ \twomat{a}{b}{0}{1} \st a\in \bbF^\times, b\in \bbF \right\} \]
equipped with the natural topology.

In this section we give an exposition of the result stated in Section~\ref{s:intro} as Theorem~\ref{t:L1-not-df}: \textit{neither $L^1(\Raff)$ nor $L^1(\Caff)$ are \qdf}.
By some standard measure-theoretic arguments,
the map $\lm_p:C_c(G)\to \PF_p(G)$ extends to a continuous, \emph{injective} algebra homomorphism $L^1(G)\to \PF_p(G)$.
Theorem~\ref{t:L1-not-df} therefore implies that if $G=\Raff$ or $G=\Caff$, none of the algebras $\PF_p(G)$ are directly finite.

We start with some general definitions. A $*$-algebra $A$ is said to be \dt{symmetric} if $\sigma(x^*x) \subseteq [0,\infty)$ for all $x \in A$, and \dt{hermitian} if $\sigma(h)\subseteq \Real$ for all self-adjoint $h\in A$. (The spectrum is taken relative to $A$ or to $\fu{A}$, depending on whether $A$ has an identity element.)
The two notions coincide for Banach $*$-algebras (see \cite[Theorem 11.4.1]{Palmer2}).
A locally compact group $G$ is said to be \dt{hermitian} if $L^1(G)$, equipped with the usual involution, is a hermitian Banach $*$-algebra. Further background on Hermitian groups can be found in \cite[\S12.6.22]{Palmer2}.

\begin{thm}[Leptin]
$\Raff$ and $\Caff$ are Hermitian.
\end{thm}

\begin{proof}
This follows from a slightly more general result \cite[Satz 6]{Leptin_SympMath22}.
\end{proof}

Barnes observed that groups which are both hermitian and amenable have the following spectral permanence property.

\begin{prop}[Barnes]\label{p:sp-perm}
Let $G$ be a hermitian, amenable group, and let $\theta: L^1(G) \to \Bdd(\cH)$ be a faithful $*$-representation on some Hilbert space. If $h=h^* \in L^1(G)$, then $\sigma_{L^1(G)}(h) = \sigma_{\Bdd(\cH)}\theta(h)$.
\end{prop}

\begin{proof}
This is essentially contained in \cite[Theorem 6]{Bar_PEMS90}.
Strictly speaking, the statement of \cite[Theorem 6]{Bar_PEMS90} only gives
$\sigma_{L^1(G)}(h) = \sigma_{\Cst_r(G)} \lambda_2(h)$.
However, examination of the proof, together with the theorem of Hulanicki that is quoted in \cite[p.~329]{Bar_PEMS90}, shows that $\lm_2$ can be replaced by any faithful $*$-representation of $L^1(G)$ on Hilbert space.
(The key point is that since $G$ is amenable, the reduced and full group $\Cst$-algebras coincide.)
\end{proof}

\paragraph{Calculations of Diep and of Rosenberg}
We first treat the case $G=\Raff$, for which our main source is Chapter 3 of the monograph \cite{Diep_book}.
The results originally date from Diep's thesis, and are summarized without full proofs in \cite{Zep_FunkPri74}.

The irreducible, continuous, unitary representations of $\Raff$ were first worked out by Gelfand and Naimark. We focus on one in particular:
let $\cH_S$ be the Hilbert space $L^2(\Real^\times, |x|^{-1}\,dx)$, and define a strongly continuous unitary representation $S: \Raff \to \Bdd(\cH_S)$ by
\[ (S_gf)(x)  = e^{ibx}f(ax) \qquad\text{for $f\in \cH_S$ and $g=\left(\begin{matrix} a & b \\ 0 & 1
    \end{matrix} \right)$.} 
\]
The representation $S$ is quasi-equivalent to the left regular representation, and so in particular the induced $*$-homomorphism $\Cst(\Raff)\to \Bdd(\cH_S)$ is injective.
(See \cite[proof of Lemma~3.3]{Diep_book}.)

\begin{thm}[Diep]
\label{t:Diep_Aff-R}
There exists $f\in L^1(\Raff)$, such that $I-S(f): \cH_S \to \cH_S$ is injective with closed range of codimension one.
\end{thm}

For the reader's convenience we give an outline of the proof.
\begin{proof}
Define $h \in L^1(\Raff)$ by
\begin{equation}\label{eq:Diep-function} 
h\left(\twomat{a}{b}{0}{1}\right) = \chi_{\abs{a}\leq 1} \frac{2a^2}{\sqrt{2\pi}} \exp (-b^2/2) \;.
\end{equation}
By \cite[Lemmas 3.6 and 3.7]{Diep_book}, $I-S(h)$ is invertible modulo the compacts, i.e.~it is a Fredholm operator on $L^2(\cH_S)$. In particular, it has closed range. Furthermore, by \cite[Lemmas 3.9 and~3.10]{Diep_book}, $I-S(h)$ has $1$-dimensional kernel and is surjective. We therefore take $f=h^*\in L^1(\Raff)$; by the previous observations, $I-S(f) = (I-S(h))^*$ is injective with closed range, and has one-dimensional cokernel.
\end{proof}

\begin{rem}
We have only taken from Diep's work what is needed for the present article. For a fuller discussion of how $\Cst(\Raff)$ arises as an extension of an abelian $\Cst$-algebra by $\Cpct(\cH_S)$, and the role played by the Fredholm operator $I-S(h)$ in classifying $\Cst(\Raff)$, see~\cite[\S3]{Diep_book}.
\end{rem}

In the case $G=\Caff$, we have a result analogous to Theorem~\ref{t:Diep_Aff-R}, obtained by Rosenberg \cite{Ros_Pac76} using a suitable adaptation of Diep's methods.
Let $\cH_S$ be the Hilbert space $L^2(\Cplx^\times,\abs{z}^{-1}\,dz)$, and define a strongly continuous unitary representation $S: \Caff \to \Bdd(\cH_S)$ by
\[ (S_gf)(z)  = e^{i \Re wz}f(az) \qquad\text{for $f\in \cH$ and $g=\left(\begin{matrix} a & w \\ 0 & 1
    \end{matrix} \right)$.} 
\]

\begin{thm}[Rosenberg]
\label{t:Rosenberg_Aff-C}
There exists $f\in L^1(\Caff)$ such that $I-S(f):\cH_S \to \cH_S$ is injective with closed range of codimension one.
\end{thm}

We omit the example and proof, which can be found, modulo some small adjustments, in \cite[Proposition 1]{Ros_Pac76}.

\begin{proof}[Proof of Theorem~\ref{t:L1-not-df}]
We first treat the case of $\Raff$. Applying Theorem~\ref{t:Diep_Aff-R}, there exists a faithful unitary representation $S: \Raff\to \cH_S$, and some $f\in L^1(\Raff)$, such that $I-S(f):\cH_S \to \cH_S$ is injective with closed range of codimension one.
Since $I-S(f)$ is not invertible in $\Bdd(\cH_S)$, $f$ is not quasi-invertible in $L^1(\Raff)$. On the other hand, note that
\[ I-S(f^*\qm f) = I-S(f)^*\qm S(f) = (I-S(f))^*(I-S(f)), \]
which \emph{is} invertible in $\Bdd(\cH_S)$.

Now, since $\Raff$ is a Hermitian group and $S$ is faithful,
by Proposition~\ref{p:sp-perm} we have
\[ \sigma_{\fu{L^1(\Raff)}} (f^*\qm f) = \sigma_{\Bdd(\cH_S)} S(f^*\qm f)\;. \]
In particular, $f^*\qm f$ is quasi-invertible in $L^1(\Raff)$, with quasi-inverse~$h$,~say. Then $h\qm f^*\in L^1(\Raff)$ and $h\qm f^*\qm f=\id$, showing that $L^1(\Raff)$ is not \qdf.

The proof for $\Caff$ is exactly similar, except that we use Theorem~\ref{t:Rosenberg_Aff-C} instead of Theorem~\ref{t:Diep_Aff-R}.
\end{proof}

\end{section}

\begin{section}{Concluding thoughts}
Theorems~\ref{t:PF2_unimodular} and~\ref{t:L1-not-df} immediately suggest the natural question:

\begin{qu}
For which locally compact groups $G$ is $\Cst_r(G)$ \qdf?
\end{qu}

Even in the special case where $G$ is a solvable Lie group, Example~\ref{eg:cheat} suggests that a full characterization may be somewhat tricky to obtain.

\begin{qu}\label{qu:if-only}
Is the completion of a \qdf\ normed algebra itself \qdf?
\end{qu}

We suspect not, but know of no counterexample. On the other hand, if Question~\ref{qu:if-only} has a \emph{positive} answer, then $\PF_p(G)$ is \qdf\ for every unimodular group and all $p\in (1,\infty)$.
\end{section}

\subsection*{Acknowledgements}
The present paper has a somewhat tangled history, and has taken shape over the course of several years while the author has held positions at Universit\'e Laval, the University of Saskatchewan, and presently Lancaster University. He thanks the Departments of Mathematics and Statistics in each of these institutions for their support, and also acknowledges 
the hospitality of the University of Leeds during a visit in May 2010, as part of a working semester on {\it Banach algebra and operator space techniques in topological group theory.}
The work done at the University of Saskatchewan was partly supported by 
NSERC Discovery Grant 402153-2011.

Some of the results presented here were included in an old preprint of the author, ``{\it Group $\Cst$-algebras which are quasi-directly finite}\/'', arXiv {\tt 1003.1650}. That preprint, which remains unpublished, is now superseded by the present article.

Thanks are due to various colleagues and correspondents, in particular C.~Zwarich for exchanges in 2008 which sowed the seeds for this line of investigation, J.~Meyer for pointing out the reference~\cite{Kap_AJM53}, and E.~Samei for lending the author his copy of~\cite{Der_book}.
The diagrams in this paper were prepared using Paul Taylor's \texttt{diagrams.sty} macros.

\appendix
\begin{section}{Ingredients in the proof of Theorem~\ref{t:pfp-in-pf2}}
\label{app:Herz-by-hand}
To fix notation and provide background, we quickly summarize the necessary definitions from \cite{Herz_p-space}. Throughout, $p\in (1,\infty)$ and $q$ is the conjugate index to~$p$.
Fix a left Haar measure $\mu$ on~$G$. Define a contractive linear map $\theta_p: L^p(G)\ptp L^q(G) \to C_0(G)$ by
\[ \theta_p(f\otimes g)(x) = \pair{\lm_p(x)f}{g} = \int_G f(x^{-1}y)g(y) \,d\mu(y) \qquad(f\in L^p(G), q\in L^q(G), x\in G), \]
and let $\FA_p(G)\subseteq C_0(G)$ be the \emph{coimage} of $\theta_p$ (more explicitly, the image of $\theta_p$ equipped with the quotient norm induced from $\left(L^p(G)\ptp L^q(G)\right) / \ker\theta_p$).
Restricting $\theta_p$ to the dense subspace $C_c(G)\tp C_c(G)$ yields a linear map $\theta: C_c(G) \tp C_c(G) \to C_c(G)$, whose image we denote by $\cA_c(G)$. Note that $\cA_c(G)$ is dense in $\FA_p(G)$, for every~$p$.

Now we make some preliminary observations.
Integration on $G$ defines a pairing $C_c(G) \times C_0(G)\to \Cplx$ and hence gives natural maps $C_c(G) \to C_0(G)^* \to \FA_p(G)^*$, the second map being restriction. We may thus regard $\theta_p^*$ as a map $C_c(G) \to (L^p(G)\ptp L^q(G))^* \iso \Bdd(L^p(G))$. Explicitly, given $f\in C_c(G)$ and $\xi\in L^p(G)$, $\eta\in L^q(G)$, we have
\[ \pair{\theta_p^*(f)\xi}{\eta} = \int_G f(t) \theta_p(\xi\tp\eta)(t)\,d\mu(t). \]
Moreover, if $\xi,\eta\in C_c(G)$ then it is easily checked that
$\pair{\theta_p^*(f)\xi}{\eta}=\pair{\lm_p(f)\xi}{\eta}$, and since both $\theta_p^*(f)$ and $\lm_p(f)$ are bounded operators, density implies
$\theta_p^*(f)=\lm_p(f)$.
In particular, $\lm_p(f)\in \theta_p^*(\FA_p(G))$ for all $f\in C_c(G)$.
Finally, since $\theta_p^* :\FA_p(G)^* \to \Bdd(L^p(G))$ is the adjoint of a quotient map, it is an isometry with weak-star closed range. Hence, $\theta_p(\FA_p(G))^*\supseteq \PF_p(G)$, so there is an isometry $\phi_p: \PF_p(G) \to \FA_p(G)^*$ such that $\theta_p^*\phi_p$ is just the inclusion of $\PF_p(G)$ into $\Bdd(L^p(G))$.


The key result needed from \cite{Herz_p-space} is the following theorem, whose proof we omit.
It should be emphasized that the theorem is independent of the results in \cite{Herz_AIF73}, which is not always made clear in the secondary literature.
\begin{thm}[{\cite[Theorem 1]{Herz_p-space}}] 
\label{t:Ap-an-A2-module}
Let $G$ be a locally compact group and let $p\in (1,\infty)$. Then the map
$\FA_p(G)\ptp \FA_2(G) \to C_0(G)$ defined by
multiplication of functions takes values in $\FA_p(G)$, and
\begin{equation}\label{eq:A_p(G)-is-A(G)-module}
\norm{fh}_{\FA_p(G)} \leq \norm{f}_{\FA_2(G)} \norm{h}_{\FA_p(G)}
\end{equation}
for all $f\in \FA_p(G)$ and $h\in\FA_2(G)$.
\end{thm}

We also need a result that is stated and proved in \cite[\S9]{Herz_AIF73}, where it is described as being folklore.

\begin{lem}\label{l:BAI-for-Ap}
Let $G$ be amenable and let $p\in (1,\infty)$. Given any compact set $K\subseteq G$, and any $\veps>0$, there exists a compactly supported function $f\in \FA_p(G)$ which is identically $1$ on~$K$ and has $\FA_p(G)$-norm at most $1+\veps$.
\end{lem}

\begin{proof}[Proof of Theorem~\ref{t:pfp-in-pf2}]
Let $h\in \FA_2(G)$ have compact support. Taking $K=\supp(h)$ in Lemma~\ref{l:BAI-for-Ap} and using \eqref{eq:A_p(G)-is-A(G)-module}, we get $h\in \FA_p(G)$ and $\norm{h}_{\FA_p(G)} \leq (1+\veps)\norm{h}_{\FA_2(G)}$, for all $\veps>0$. Since $\cA_c(G)$ is norm-dense in $\FA_2(G)$, this implies that $\FA_2(G) \subseteq \FA_p(G)$ with non-increase of norms.
Moreover, since $\imath(\cA_c(G))=\cA_c(G)$, and $\cA_c(G)$ is dense in $\FA_p(G)$, the inclusion $\imath: \FA_2(G) \to \FA_p(G)$ has dense range. Hence $\imath^*:\FA_p(G)^* \to \FA_2(G)^*$ is injective, and has norm $\leq 1$.

Now consider the following commutative diagram:
\[ \begin{diagram}[tight,height=3em,width=5em]
C_c(G) & \rTo^{\lm_p} & \PF_p(G) & \rTo^{\phi_p} & \FA_p(G)^* \\
 &\rdTo_{\lm_2} & & & \dTo^{\imath^*} \\
 & & \PF_2(G) & \rTo_{\phi_2} & \FA_2(G)^*
\end{diagram} \]
For each $f\in C_c(G)$,
$ \norm{\lm_2(f)} = \norm{ \phi_2\lm_2(f)}
 = \norm{\imath^* \phi_p\lm_p(f)} \leq \norm{\lm_p(f)}$.
Therefore, since $\lm_p(C_c(G))$ is norm-dense in $\PF_p(G)$, there is a unique continuous linear map $J: \PF_p(G)\to\PF_2(G)$ such that $J\lm_p=\lm_2$. By continuity, $J$ is an algebra homomorphism.
Moreover, for each $f\in C_c(G)$,
\[ \phi_2 J \lm_p(f) = \phi_2  \lm_2(f) = \imath^*\phi_p \lm_p(f) \]
so by density, $\phi_2 J = \imath^*\phi_p$. Since $\imath^*\phi_p$ is injective, so is $J$, and the proof of Theorem~\ref{t:pfp-in-pf2} is complete.
\end{proof}

\end{section}

\newcommand{\arX}[1]{{\tt arXiv} #1}
\newcommand{\url}[1]{\texttt{#1}}  

\bibliography{df-pfp}

\def\cprime{$'$}
\begin{thebibliography}{10}

\bibitem{Bar_PEMS90}
{\sc B.~A. Barnes}, {\em When is the spectrum of a convolution operator on
  {$L^p$} independent of $p$?}, Proc. Edinburgh Math. Soc. (2), 33 (1990),
  pp.~327--332.

\bibitem{BonsDunc}
{\sc F.~F. Bonsall and J.~Duncan}, {\em Complete normed algebras}, Ergebnisse
  der Mathematik und ihrer Grenzgebiete, Band 80, Springer-Verlag, New York,
  1973.

\bibitem{YC_surjunc}
{\sc Y.~Choi}, {\em Group representations with empty residual spectrum}, Int.
  Eq. Op. Th., 67 (2010), pp.~95--107.
\newblock Erratum: Int. Eq. Op. Th. 69 (2011), no. 1, 149--150.

\bibitem{Cowling_KS78}
{\sc M.~Cowling}, {\em The {K}unze-{S}tein phenomenon}, Ann. Math. (2), 107
  (1978), pp.~209--234.

\bibitem{Der_book}
{\sc A.~Derighetti}, {\em Convolution operators on groups}, vol.~11 of Lecture
  Notes of the Unione Matematica Italiana, Springer, Heidelberg, 2011.

\bibitem{Diep_book}
{\sc D.~N. Diep}, {\em Methods of noncommutative geometry for group {$C\sp
  *$}-algebras}, vol.~416 of Chapman \& Hall/CRC Research Notes in Mathematics,
  Chapman \& Hall/CRC, Boca Raton, FL, 2000.

\bibitem{Dix_VN_fr2}
{\sc J.~Dixmier}, {\em Les alg\`ebres d'op\'erateurs dans l'espace hilbertien
  (alg\`ebres de von {N}eumann)}, Gauthier-Villars \'Editeur, Paris, 1969.
\newblock Deuxi{\`e}me {\'e}dition, revue et augment{\'e}e, Cahiers
  Scientifiques, Fasc. XXV.

\bibitem{Dix_C*_en}
\leavevmode\vrule height 2pt depth -1.6pt width 23pt, {\em {$C\sp*$}-algebras},
  North-Holland Publishing Co., Amsterdam, 1977.
\newblock Translated from the French by Francis Jellett, North-Holland
  Mathematical Library, Vol. 15.

\bibitem{Herz_p-space}
{\sc C.~Herz}, {\em The theory of {$p$}-spaces with an application to
  convolution operators.}, Trans. Amer. Math. Soc., 154 (1971), pp.~69--82.

\bibitem{Herz_AIF73}
\leavevmode\vrule height 2pt depth -1.6pt width 23pt, {\em Harmonic synthesis
  for subgroups}, Ann. Inst. Fourier (Grenoble), 23 (1973), pp.~91--123.

\bibitem{Kap_AJM53}
{\sc I.~Kaplansky}, {\em Modules over operator algebras}, Amer. J. Math., 75
  (1953), pp.~839--858.

\bibitem{KunzeStein_AJM60}
{\sc R.~A. Kunze and E.~M. Stein}, {\em Uniformly bounded representations and
  harmonic analysis of the {$2\times 2$} real unimodular group}, Amer. J.
  Math., 82 (1960), pp.~1--62.

\bibitem{Leptin_SympMath22}
{\sc H.~Leptin}, {\em Lokal kompakte {G}ruppen mit symmetrischen {A}lgebren},
  in Symposia {M}athematica, {V}ol.~{XXII} ({C}onvegno sull'{A}nalisi
  {A}rmonica e {S}pazi di {F}unzioni su {G}ruppi {L}ocalmente {C}ompatti,
  {INDAM}, {R}ome, 1976), Academic Press, London, 1977, pp.~267--281.

\bibitem{Lohoue_1980}
{\sc N.~Lohou{\'e}}, {\em Estimations {$L^{p}$} des coefficients de
  repr\'esentation et op\'erateurs de convolution}, Adv. in Math., 38 (1980),
  pp.~178--221.

\bibitem{MO16944}
{\sc J.~Meyer}, {\em personal communication}.
\newblock MathOverflow.
\newblock \url{http://mathoverflow.net/questions/16944} (version: 2010-03-09).

\bibitem{Mon_dirfin}
{\sc M.~S. Montgomery}, {\em Left and right inverses in group algebras}, Bull.
  Amer. Math. Soc., 75 (1969), pp.~539--540.

\bibitem{Munn_DF1}
{\sc W.~D. Munn}, {\em Direct finiteness of certain monoid algebras.~{I}},
  Proc. Edinburgh Math. Soc. (2), 39 (1996), pp.~365--369.

\bibitem{Nebbia_KS88}
{\sc C.~Nebbia}, {\em Groups of isometries of a tree and the {K}unze-{S}tein
  phenomenon}, Pacific J. Math., 133 (1988), pp.~141--149.

\bibitem{Palmer2}
{\sc T.~W. Palmer}, {\em Banach algebras and the general theory of
  {$*$}-algebras. {V}ol. 2}, vol.~79 of Encyclopedia of Mathematics and its
  Applications, Cambridge University Press, Cambridge, 2001.

\bibitem{Ped_cstar-book}
{\sc G.~K. Pedersen}, {\em {$C^{\ast} $}-algebras and their automorphism
  groups}, vol.~14 of London Mathematical Society Monographs, Academic Press
  Inc. [Harcourt Brace Jovanovich Publishers], London, 1979.

\bibitem{Ros_Pac76}
{\sc J.~Rosenberg}, {\em The {$C\sp*$}-algebras of some real and {$p$}-adic
  solvable groups}, Pacific J. Math., 65 (1976), pp.~175--192.

\bibitem{Wang_pitman199}
{\sc X.~Wang}, {\em The {$C^*$}-algebras of a class of solvable {L}ie groups},
  vol.~199 of Pitman Research Notes in Mathematics Series, Longman Scientific
  \& Technical, Harlow, 1989.

\bibitem{Zep_FunkPri74}
{\sc D.~N. Z{\cprime}ep}, {\em The structure of the group {$C$}*-algebra of the
  group of affine transformations of the line}, Funkcional. Anal. i Prilo\v
  zen., 9 (1974), pp.~63--64.

\end{thebibliography}
\bibliographystyle{siam}

\vfill

\noindent
\begin{tabular}{l}
\contact
\end{tabular}

\end{document}